\documentclass[a4paper]{article}
\usepackage[T1]{fontenc}
\usepackage[utf8]{inputenc}
\usepackage{amsmath,amssymb,mathtools,bm,amsthm,url,xparse,xcolor}
\usepackage{graphicx}
\usepackage{float,anyfontsize}
\usepackage{nicefrac}
\usepackage{pgf}
\usepackage{pgffor}
\usepgflibrary{plothandlers}
\usepackage{listings}
\usepackage[english,danish,american]{babel}
\usepackage{amsfonts}
\usepackage[square, numbers]{natbib}
\usepackage{enumerate}
\usepackage{bbm}
\usepackage[mathscr]{euscript}
\usepackage{stmaryrd}
\usepackage{soul}
\usepackage{hyperref}
\usepackage{xspace}
\usepackage[draft,danish]{fixme}
\usepackage{mathtools}
\usepackage{url}
\usepackage[ps,all,dvips]{xy}

\usepackage[shortlabels]{enumitem}
\setlist[enumerate]{
  label=\textnormal{(\alph*)},
}

\SetEnumitemKey{roman}{ label=\textnormal{(\roman*)} }
\SetEnumitemKey{arabic}{ label=\textnormal{(\arabic*)} }

\usepackage{interval}
\intervalconfig{soft open fences}

\DeclarePairedDelimiter\abs\lvert\rvert

\DeclarePairedDelimiterX\Set[1]\{\}{
  
  #1}

\DeclarePairedDelimiterX\BBase[1][]{
  
  #1}
\DeclarePairedDelimiterX\PBase[1](){
  
  #1}

\newcommand\Aver{\mathbb{E}\BBase}

\newtheorem{theorem}{Theorem}[section]
\newtheorem{lemma}[theorem]{Lemma}
\newtheorem{definition}[theorem]{Definition}

\theoremstyle{definition}

\newcommand\ff{\mathcal{F}}

\DeclareDocumentCommand\ff{ t. }{
  \IfBooleanTF{#1}{
    {\mathcal{F}\bm{.}}
  }{
    \mathcal{F}
  }
}

\newcommand\N{\mathbb{N}}
\newcommand\R{\mathbb{R}}

\newcommand\Z{\mathbb{Z}}
\newcommand\C{\mathbb{C}}

\newcommand\dif{\,d}


\newcommand\altadformatter[1]{\textit{#1}:}
\newcommand\setAdtext[1]{%
  \def\@adtext{#1}%
}
\setAdtext{Betingelse}
\newcommand\Adimplies[2]{
  \par%
  \penalty-1000
  \medskip\noindent%
  \begingroup%
  \let\textup\relax%
  \altadformatter{\@adtext~\ref{#1} $\Rightarrow$ 
    \MakeLowercase{\@adtext}~\ref{#2}}%
  \endgroup%
  \enspace%
}
\makeatother

\usepackage{dsfont} 

\begin{document}


\title{Local asymptotic self-similarity for heavy tailed\\ harmonizable fractional  L\'evy motions}



\author{Andreas Basse-O'Connor, Thorbjørn Grønbæk and Mark Podolskij \\
Department of Mathematics, Aarhus University \\
\{basse, thorbjoern, mpodolskij\}@math.au.dk}



\maketitle

\begin{abstract}
In this work we characterize the local asymptotic self-similarity of harmonizable fractional L\'evy motions 
in the heavy tailed case. The corresponding tangent process is shown to be the harmonizable fractional stable motion.  In addition, we  provide sufficient conditions for existence of harmonizable fractional L\'evy motions. 
%
\\
\\
\noindent \textit{Keywords: local asymptotic self-similarity; harmonizable processes; fractional processes; spectral representations.}


\end{abstract}



\numberwithin{equation}{section} 
\pagestyle{empty}
\pagestyle{plain}

\section{Introduction}

The class of self-similar stochastic processes plays a key role  in probability theory as they appear in some of the most fundamental limit theorems, see \cite{lamperti1962semi}, and in modeling they are used  in geophysics, hydrology, turbulence and economics, see \cite{willinger1996bibliographical} for numerous references. This class of stochastic processes are invariant in distribution under suitable time and space scaling, that is, a stochastic process $(X_t)_{t\in\R}$ is called self-similar with index $H\in\R$ if for all $c>0$ the two processes   $(X_{ct})_{t\in \R}$ and $(c^H X_t)_{t\in \R}$ equals in finite dimensional distributions. 
The only self-similar centered Gaussian process with stationary increments is the fractional Brownian motion (up to scaling), which is a centered Gaussian process $(X_t)_{t\in \R}$ with $X_0=0$ a.s.\  and covariance function 
\begin{equation}\label{sfdsfswfh}
\mathrm{Cov}(X_t,X_s) = \frac{1}{2}(|t|^{2H}+|s|^{2H}-|s-t|^{2H}) \qquad \text{for all } s,t\in \R,
\end{equation}
where $H\in (0,1)$. The fractional Brownian motion has  
a \emph{moving average representation} and a \emph{harmonizable representation}, and both lead to the same process (defined by \eqref{sfdsfswfh}), for further details see Subsection~\ref{sdlfjhhs}.  However for non-Gaussian processes their moving average and harmonizable representations are very different, see e.g.\ \cite{Camb1989} and \cite{SamoTaqqu1994} for the case of $\alpha$-stable processes. Only a very specific class of processes are exact self-similar, but a much larger class of processes behaves locally as a self-similar processes - this is already seen within the class of L\'evy processes. 

A stochastic process $(X_t)_{t\in\R}$ is said to be  \emph{locally asymptotically self-similar}  if  there exists a number $H\in \R$ and a non-degenerate process $(T_t)_{t\in \R}$ such that
\begin{align}\label{slkjdfljsd}
\PBase*{\frac{X_{\epsilon t}}{\epsilon^H}}_{t\in\R}  \xrightarrow[\epsilon \to 0_+]{d} (T_t)_{t\in\R},
\end{align}
where $\xrightarrow{d}$ denotes converege in finite dimensional distributions.
The process  $T=(T_t)_{t\in\R}$ is called the tangent process of $X$, and by \eqref{slkjdfljsd}, $T$  is necessarily self-similar.  
Local self-similarity means that at small time-scales the stohastic process $(X_t)_{t\in\R}$ is approximately self-similar and may be approximated by its tangent process. This property was introduced to provide a more flexible modeling framework compared to global self-similarity.  For applications, it has been used to  study the behaviour of flows, see \cite{Crovella1996} and \cite{stoev2006lass}, and for showing high frequency asymptotic results,  see
\cite{Bardet-Surgailis} or \cite{BLP}.

\medskip
\noindent
 \textbf{Moving average fractional L\'evy motions:}
Starting from the moving average representation of the fractional Brownian motion, 
 \citep{Marquardt2006} has, among many others, studied fractional L\'evy processes defined as 
 \begin{equation}\label{slfj}
 X_t =  \int_{-\infty}^t \Big((t-s)_+^{\beta} - (-s)_+^\beta\Big)\,dL_s,\qquad t\in \R, 
 \end{equation}
 where $\beta\in (0,1/2)$ and $(L_t)_{t\in \R}$ is a centered L\'evy process 
 with finite second moment. Throughout this paper  $x_+:=\max\{x,0\}$ and $x_-:=-\min\{x,0\}$ denote the  positive and negative parts of any number $x\in\R$. 
 
In the following, we will call such processes for \emph{moving averages fractional L\'evy motions} to distinct them from their harmonizable counterpart. Under a regular variation assumption on the L\'evy measure of $L$ near zero,    \cite{Marquardt2006}  shows that a moving average fractional L\'evy motion is never self-similar, but it is locally asymptotically self-similar with tangent process the linear fractional stable motion, which is a process of the form \eqref{slfj} with $L$ being an $\alpha$-stable L\'evy process,
  cf.\ \cite{Camb1989} and Theorems~4.4 and 4.5 of \cite{Marquardt2006}.

 \medskip
\noindent
 \textbf{Harmonizable fractional L\'evy motions:}
 Next we define the class of harmonizable fractional Lévy motions which includes the harmonizable fractional stable motion  introduced in \cite{Camb1989}.

\begin{definition} \label{def:hflm}
A stochastic process $(X_t)_{t\in\R}$ is called a harmonizable fractional Lévy motion  with parameters 
$(\alpha,H)\in \R^2$  if 
\begin{align}\label{def-hflm}
X_t=\int_\R \frac{e^{its}-1}{is} \PBase*{a(s_+)^{-H-1/\alpha+1} + b(s_-)^{-H-1/\alpha+1}} \dif L_s, \quad t\in\R
\end{align}
where $L$ is a rotationally invariant complex-valued Lévy process, and $a,b\in \R$. 
\end{definition}

The (over) parametrization in Definition~\ref{def:hflm} is chosen due to our forthcoming Assumption~(A). In fact under Assumption~(A) below, the $H$ parameter in Definition~\ref{def:hflm} turns out to be exactly the number $H$ in the definition of local asymptotic self-similarity.  From Theorem~\ref{thm:hflm}, below, it follows  that the harmonizable fractional Lévy motions have stationary increments and  rotational invariant distributions. Furthermore, we give concrete conditions for existence of the harmonizable fractional L\'evy motion on $(\alpha, H)$ and the L\'evy measure of $L$. 

In \cite{Benassi2002}, local asymptotic self-similarity is studied for a slightly different class of harmonizable fractional motions  under the assumption that all moments are finite, e.g.\ the Lévy measure $\nu$ of the Lévy process $L$ satisfies that
 \begin{equation}\label{dsfkhbhs}
 \int_{|x|>1} |x|^p\,\nu(dx)<\infty\qquad \text{ for all } p>0.
 \end{equation}
Their result is the following: 
\begin{theorem}[Benassi, Cohen and  Istas] \label{thmBCI}
Let $(X_t)_{t\in \R}$ denote a harmonizable fractional L\'evy motion as in Definition~2.3 of \cite{Benassi2002} satisfying the moment condition (\ref{dsfkhbhs}). Then the process $X$ is locally asymptotically self-similar with index $H$ and tangent process the fractional Brownian motion, that is, 
\begin{align}
\label{BCI}
\PBase*{\frac{X_{\epsilon t}}{\epsilon^H}}_{t\in\R} \xrightarrow[\epsilon \to 0_+]{d} (c_0 B^H_t)_{t\in\R}.
\end{align}
where $(B^H_t)_{t\in \R}$ is a fractional Brownian motion with Hurst index $H$ and $c_0$ is a suitable constant. 
\end{theorem}

    The main aim of this work is to characterize the local asymptotic  self-similarity of  the harmonizable fractional L\'evy motion   when $L$ has heavy tails, violating the moment condition \eqref{dsfkhbhs}. 
The methods of \cite{Benassi2002}  rely heavily on  power series expansion of the characteristic function which is only available under the assumption~(\ref{dsfkhbhs}).
%
Instead of this assumption, we consider the case where the 
L\'evy measure $\nu$ is regular varying in the following sense. 

\medskip\noindent
 \textbf{Assumption (A):}  \emph{Suppose that  $L$ is a rotationally invariant complex-valued L\'evy process without Gaussian component and let $\nu$ denote its L\'evy measure. We assume that  $\nu$ is absolutely continuous with respect to the two dimensional Lebesgue measure with a density $f:\R^2\to \R_+$ satisfying 
\begin{align*}
{}& f(x)\sim c_0\Vert x \Vert^{-2-\alpha}\quad \; \,\text{  as } \Vert x \Vert \to \infty \qquad \text{ and}
\\
{}& f(x) \leq C\Vert x \Vert^{-2-\alpha} \quad \text{for all }x\in\R^2,
\end{align*}
where $c_0, C>0$. }

\medskip
The following theorem, which  is the main result of this paper, characterizes the local asymptotic 
self-similarity of harmonizable fractional L\'evy motions in the heavy-tailed case, and additional provides an existence result  for them.

\begin{theorem} \label{thm:lass}
Let  $(\alpha,H)\in \interval[open]{0}{2}\times\interval[open]{0}{1}$ and suppose that Assumption~(A) is satisfied.   Then the  harmonizable fractional L\'evy motion  $(X_t)_{t\in\R}$,  defined in \eqref{def-hflm}, is well-defined and it is locally asymptotically self-similar with index $H$ and tangent process the harmonizable fractional stable motion, that is, \begin{align} \label{align:LASSpropertyofHFLM}
\PBase*{\frac{X_{\epsilon t}}{\epsilon^H}}_{t\in\R} \xrightarrow[\epsilon \to 0_+]{d} (C_t)_{t\in\R},
\end{align}
where the convergence is in finite dimensional distributions and $(C_t)_{t\in\R}$ denotes a harmonizable fractional stable motion with parameters $(\alpha,H)$, which is defined in \eqref{def-hflm} with $L$ being a complex-valued rotationally invariant $\alpha$-stable L\'evy process. 
\end{theorem}
The choice of constants for $(C_t)_{t\in\R}$ can be found by examining the proof. We note that the tangent process in Theorem~\ref{thm:lass} differs from the tangent processes appearing in Theorem~\ref{thmBCI} and Theorem~4.5 of \cite{Marquardt2006}. From this we infer that it is the behaviour of the L\'evy measure of $L$ close to zero which dominates  in the moving average setting, whereas it is the behaviour of the L\'evy measure of $L$ far away from zero which dominates  in the harmonizable setting. 
The structure of the paper is as follows:  Section~\ref{sdlfjhhs} explains the role played by harmonizable processes within the class of stationary processes.  Section \ref{sect:harmproc} introduces complex random measures, their integration and provide existence criterias for  harmonizable processes. Finally, at the end of the last section, we present the proof of Theorem~\ref{thm:lass}. 

\section{Background on harmonizable processes}\label{sdlfjhhs}

Stationary processes are one of the main classes of stochastic processes. For stationary, centered Gaussian processes, it is well-known that every $L^2$-continuous process $(X_t)_{t\in \R}$ has a \emph{harmonizable representation} of the form 
\begin{align}\label{dslfjsdf}
X_t = \int_\R e^{its} \, M(ds), \quad t\in \R,
\end{align}
for some complex-valued Gaussian random measure $M$ defined on $\R$. Furthermore, a rather large class of these processes have, in addition, a \emph{moving average representation}, that is, a representation of the form 
\begin{align}\label{sflsf}
X_t = \int_{\R} g(t-s)\,  dB_s, \quad t\in\R,
\end{align}
where $g$ is a deterministic function and $(B_t)_{t\in \R}$ is a two-sided real-valued Brownian motion. (Note that, the Brownian motion may be viewed as a shift-invariant Gaussian random measure.) Indeed, the class of Gaussian processes having a moving average representation corresponds exactly to those processes with absolute continuous spectral measure $\mu$. Recall that the spectral measure $\mu$ is given by $\mu(A)=\mathbb E[ | M(A)|^2]$ for $A\in \mathcal B(\R)$, where $M$ is given in \eqref{dslfjsdf}.    These classical results can be found in e.g.\ \cite{doob1953stochastic} or \cite{yaglom2004introduction}.

The only centered  Gaussian self-similar process with stationary increments is the  fractional Brownian motion $(B^H_t)_{t\in \R}$ with Hurst index $H\in (0,1)$,  and as already mentioned in the introduction, this process has the following two representations 
\begin{align*}
B^H_t={}& \int_{\R} \PBase*{(t-s)_+^{H-1/2} - (-s)_+^{H-1/2}}\dif B_s, \ \ \text{(``moving average representation'')}
\\
B^H_t={}& \int_{\R} \frac{e^{its}-1}{is} \abs{s}^{-H-1/2+1}  \,M(ds), \qquad   \quad \text{(``harmonizable representation'')},
\end{align*} 
which yields the same process in distribution, see Chapter~7.2 of \cite{SamoTaqqu1994} for further details. Hence, the fractional Gaussian noise $(B^H_n-B^H_{n-1})_{n\in \Z}$ has  both a harmonizable, \eqref{dslfjsdf}, and a moving average,  \eqref{sflsf},  representation. For comparison we will discuss the structure of stationary $\alpha$-stable processes with $\alpha\in (0,2)$ in the following.  

In sharp contrast to the Gaussian situation the class of $\alpha$-stable stationary increments self-similar processes, $\alpha\in (0,2)$, is huge, and is far from being understood by now. However, two natural generalizations of the fractional Brownian motion to the $\alpha$-stable setting are proposed in \cite{Camb1989} generalizing the fractional Brownian motion to $\alpha$-stabe processes by replacing the driving Gaussian random measure with an $\alpha$-stable random measure in its moving average and harmonizable representations. This leads to the \emph{harmonizable fractional stable motion} $(X_t)_{t\in \R}$, which is defined as 
\begin{align*}
X_t=\int_\R \frac{e^{its}-1}{is}\PBase*{a(s_+)^{-H-1/\alpha +1} + b (s_-)^{-H-1/\alpha+1}}\, dL_s, \quad t\in\R,
\end{align*}
where $(L_t)_{t\in \R}$ is a two-sided, complex-valued,   $\alpha$-stable, rotationally invariant L\'evy process, and to the \emph{linear fractional stable motion} $(X_t)_{t\in \R}$, which is defined as
\begin{align*}
X_t=\int_\R a\big((t-s)_+^{H-1/\alpha} - (-s)_+^{H-1/\alpha}\big) + b \big( (t-s)_-^{H-1/\alpha} - (-s)_-^{H-1/\alpha}\big) \dif L_s,
\end{align*}
where $(L_t)_{t\in\R}$ is a two-sided, real-valued, $\alpha$-stable, symmetric Lévy process. 
Notice that  corresponding noise processes $(X_n-X_{n-1})_{n\in \Z}$ for the linear and harmonizable fractional stable motions are moving averages and harmonizable processes, respectively.

Indeed, the Gaussian assumption is crucial for the above equality between the harmonizable and moving average representations to hold, as it turns out that harmonizable fractional stable motion and linear fractional stable motion as quite different processes, cf.\ \cite{Camb1989} and \cite{SamoTaqqu1994}.
The seminal paper \cite{Rosinski1995} shows that every stationary $\alpha$-stable process has a \emph{unique} decomposition into a (mixed) moving average component, a harmonizable component and a process of the  ``third kind'', which does  not admit moving average nor harmonizable components.  The class of mixed  moving averages may be viewed as the class of processes having the least memory, whereas class of harmonizable processes is the class having the largest degree of memory, and the processes of the third kind are in  between. These facts come from ergodic consideration, see the introduction of \cite{Ros-Sam} for more details, and are also illustrated by the fact that moving averages are always  mixing and harmonizable processes are never  ergodic nor mixing. Hence by studying moving averages and harmonizable processes, we are examining the two extremes of stationary $\alpha$-stable  processes. 

Thus the comparison of results on local asymptotical self-similarity in the introduction between linear fractional stable motions and harmonizable fractional stable motions are, in fact, a comparison between  $\alpha$-stable self-similar stationary increments processes with the least memory and  with the most memory. This encircles the local asymptotical behaviour of general $\alpha$-stable self-similar processes with stationary increments.

\section{Preliminaries on complex stochastic integration theory} \label{sect:harmproc}

All random variables and processes will be defined on a probability space~$(\Omega, \mathcal{F}, \mathbb{P})$.  A real-valued stochastic variable $X$ is symmetric $\alpha$-stable ($S\alpha S$) if for some $\alpha\in\interval[open left]{0}{2}$, the characteristic function of $X$ satisfies
\begin{align*}
\mathbb{E}\BBase*{\exp\Set*{i t X}} = \exp\PBase*{-\sigma^{\alpha} \abs{t}^{\alpha}},\qquad \text{for all } t \in \R, 
\end{align*}
for some parameter $\sigma> 0$ called the scale parameter. If $\alpha=2$, then $X$ has a centered Gaussian distribution and $\sigma^2$ is the variance of $X$. Rotationally invariant random variables and processes are defined as follows:

\begin{definition}
A complex-valued random variable $X$ is \emph{rotationally invariant} if
\begin{align} \label{align:rotinv}
e^{i\theta}X\stackrel{d}{=}X, \quad \textrm{for all } \;  \theta \in \interval[open right]{0}{2\pi},
\end{align}
where $\stackrel{d}{=}$ denotes equality in distribution.
Similarly, a complex-valued stochastic process $(X_t)_{t\in T}$ is \emph{rotationally invariant} if every complex linear combination is rotationally invariant, e.g.\ $\sum_{n=1}^N z_n X_{t_n}$ is rotationally invariant.
\end{definition} 
Rotational invariance is called isotropy in some references but due to the ambiguity of isotropy we chose to use rotational invariance, cf.\ the discussion in Example 1.1.6 of \cite{SamorodnitskyBook}. A complex-valued  process can equivalently be regarded as a $\R^2$-valued random variable, in which case rotational invariance is invariance in distribution wrt.\ rotation matrices. We will with some ambiguity switch between the $\C$ and $\R^2$.
From the definition it is immediate that a rotationally invariant random variable $X=X_1+iX_2$ is symmetric and furthermore if it is infinitely divisible, then $X_1$ and $X_2$ share the same Lévy measure $\nu$.
Let $\mathcal{B}(\R)$ denote the Borel sets on $\R$, $\mathcal{B}_b(\R)$ the bounded Borel sets on $\R$ and $L^0_{\C}(\Omega)$ the complex-valued random variables. 
For completeness, we define complex-valued infinitely divisible random measures and state well known stochastic integration results, cf.\ \cite{urbanik1968random} and \cite{RajputRosinski1989}.

\begin{definition}[Complex-valued random measure] \label{definition:crm}
A complex-valued random measure is by definition a complex-valued set function
\begin{align*}
M: \mathcal{B}_b(\R) \to L^0_{\C}(\Omega),
\end{align*}
such that for disjoints sets $A_1,A_2,\ldots \in \mathcal{B}_b(\R)$, the complex-valued random variables 
\begin{align*}
M(A_1), M(A_2), \ldots
\end{align*}
are independent and infinitely divisible,  and if $\bigcup_{n\in\N} A_n \in \mathcal{B}_b(\R)$ then
\begin{align*}
M\PBase*{\bigcup_{n=1}^\infty A_n}=\sum_{n=1}^{\infty} M(A_n) \quad a.s.,
\end{align*}
where the series converges almost surely. 
\end{definition}


Given a complex-valued random measure $M$ we can find a $\sigma$-finite deterministic measure $\lambda$ on $\R$ such that $\lambda (A_n)\to 0$ implies $M(A_n)\to 0$ in probability. We call $\lambda$ a control measure for the random measure $M$. Letting $\nu_A(\cdot)$ denote the Lévy measure of $M(A)$, we can then apply Proposition~2.4 of \cite{RajputRosinski1989} to obtain a decomposition such that 
\begin{align*}
F(A\times B) \coloneqq \nu_A(B) = \int_\R \int_{\R^2} \mathds{1}_{A\times B}(s,x) \, \rho(s,dx) \lambda(ds),
\end{align*}
where $\{ \rho(s,dx)\}_{s\in\R}$ denotes a family of Lévy measures on $\R^2$. For the rest of the paper, we shall use the notation
\begin{align} \label{align:Kfunc}
K(\theta,s)\coloneqq \int_{\R^2} \BBase*{e^{i\langle \theta,x\rangle} -1 -\mathds{1}_{\Set{\Vert x \Vert \leq 1}}\langle \theta,x\rangle} \rho(s,dx), \quad (\theta,s) \in \R^2\times\R.
\end{align}

A simple complex-valued function $f:\R \to \C$ is a function of the (canonical) form
\begin{align} \label{align:simplefunction}
f(s)=\sum_{j=1}^n z_j \mathds{1}_{A_j},
\end{align}
where $n\in\N$, $z_1,\ldots, z_n$ are complex numbers and $A_1,\ldots, A_n$ are disjoint sets from $\mathcal{B}_b(\R)$. For a simple function $f$, of the form \eqref{align:simplefunction},  and $A \in \mathcal{B}(\R)$ we define
\begin{align*}
\int_A f \dif M = \sum_{j=1}^n z_j M(A\cap A_j).
\end{align*}
A (general) measurable function $f:\R \to \C$ is said to be $M$-integrable,  if there exists a sequence of simple function $\Set{f_n}_{n\in\N}$ such that
\begin{enumerate}[label=(\roman*)]
\item $f_n \to f$, $\lambda$-almost surely.
\item for every $A\in\mathcal{B}(\R)$, the sequence $\Set{\int_A f_n \dif M}_{n\in\N}$ converges in probability, as $n\to\infty$.
\end{enumerate}
In the affirmative case, we define
\begin{align*}
\int_A f \dif M \coloneqq \mathbb{P}-\lim_{n\to\infty} \int_A f_n \dif M,
\end{align*}
where $\Set{f_n}$ satisfies (i) and (ii) and $\mathbb{P}-\lim$ denotes limit in probability. It can be shown that this definition does not depend on the approximating sequence $\Set{f_n}$.
For further details on stochastic integration theory we refer to \cite{RajputRosinski1989}, \cite{SamorodnitskyBook}, \cite{SamoTaqqu1994} and \cite{urbanik1968random}. In the following $\Re(z), \Im(z)$ denotes real, respectively imaginary, part of a complex number $z$.
\begin{theorem} \label{thm:CFforcomplexint}
\ \\ 
(a): 
Let $f:\R\to\C$ be a measurable function.\ Write $f=f_1+if_2$. Then $f$ is $M$-integrable if the following condition hold true
\begin{align*}
\int_\R \int_{\R^2} \min \bigg( 1,\bigg[\Vert \Big( f_1(s)+f_2(s), \;
  f_1(s)-f_2(s) \Big) \Vert^2\bigg] \Vert x \Vert^2  \bigg) \rho(s,dx) \lambda (ds) < \infty,
\end{align*}
and, in the affirmative case, the characteristic function of $\int_{\R} f \dif M$ is given by
 \begin{align*}
& \mathbb{E} \BBase*{ \exp\PBase*{i\Set*{\theta_1 \Re(\int_\R f \dif M) + \theta_2 \Im(\int_\R f \dif M)}}}
\\ &\qquad 
 = \exp\PBase*{\int_\R K\bigg( \theta_1f_1(s)+\theta_2 f_2(s),\; \theta_2 f_1(s)-\theta_1 f_2(s)\big) ,s\bigg) \, \lambda(ds)}.
\end{align*}
(b): 
Suppose $f_1,\ldots,f_n$ are $M$-integrable. The joint characteristic function is given by
\begin{align*}
&\mathbb{E}\BBase*{\exp\Set*{i \sum_{j=1}^n \Big(\theta_j^{(1)} \Re(\int f_j \dif M) + \theta_j^{(2)} \Im(\int f_j \dif M)\Big)}}
\\ & \qquad 
= \exp\Big(\int_\R K\bigg(\sum_{j=1} \theta_j^{(1)} f_{j,1} + \theta_j^{(2)}f_{j,2}, \sum_{j=1}^n \theta_j^{(2)} f_{j,1} - \theta_j^{(1)} f_{j,2},s\bigg) \, \lambda(ds)\Big).
\end{align*}
(c): 
Let $M=M^{(1)}+iM^{(2)}$ be a rotationally invariant complex-valued random measure and let $f:\R\to\C$ be a measurable function. Then the following integrals exists simultaneously and are equal in distribution:
\begin{align}\label{sdfljsldjf}
\int_\R f \dif M \stackrel{d}{=} \int_\R \Vert {f}\Vert \dif M = \int_\R \Vert{f} \Vert \dif M^{(1)} + i \int_\R \Vert{f} \Vert \dif M^{(2)}.
\end{align}
\end{theorem}
\begin{proof}
(a) follows from the same steps as Theorem 2.7 in \cite{RajputRosinski1989} using complex-valued functions instead. (b) follows by the same steps as in the proof for Proposition 6.2.1(iii) of \cite{SamoTaqqu1994}. (c) follows by closely examining the results and arguments in \cite{urbanik1968random}.
\end{proof}
Often it is easier to think of the complex-valued stochastic integral as
\begin{align*}
\int_\R f \dif M_s =& \int_\R (f_1+if_2) \dif (M^{(1)}+iM^{(2)})
\\
=& \int_\R f_1 \dif M^{(1)} - \int_\R f_2 \dif M^{(2)} + i\PBase*{\int_\R f_1 \dif M^{(2)} + \int_\R f_2 \dif M^{(1)}}
\end{align*}
and show existence for each of the above four integrals separately (this is a more strict existence criterion). As a consequence of $(c)$, it is also necessary to prove existence of all of these four integrals, when $M$ is a rotationally invariant random measure.

\section{Existence and properties of  harmonizable fractional L\'evy motions}

Recall that a harmonizable fractional Lévy motion $(X_t)_{t\in \R}$ is defined by  
\begin{align*}
X_t=\int_\R \frac{e^{its}-1}{is} \PBase*{a(s_+)^{-H-1/\alpha+1} + b(s_-)^{-H-1/\alpha+1}} \dif L_s, \quad t\in\R,
\end{align*}
where $L$ is a rotational invariant complex-valued Lévy process. Our next result gives a general existence criterion for harmonizable fractional L\'evy motions together with some properties.

\begin{theorem} \label{thm:hflm}
Let $(L_t)_{t\in \R}$ be a complex-valued rotational invariant L\'evy process without Gaussian component. 
The harmonizable fractional Lévy motion $(X_t)_{t\in \R}$, defined in Definition~\ref{def:hflm},  with parameters $(\alpha, H)\in (0,2)\times (0,1)$ exists if both of the  following (a)--(b) are satisfied: 
\begin{enumerate}
\item $\int_{\abs{x} >1}\, \abs{x}^{\frac{1}{H+1/\alpha}} \, \nu_R (dx)<\infty$, 
\item 
$
\int_{\abs{x} \leq 1}\, \abs{x}^{\frac{1}{H+1/\alpha-1}} \, \nu_R (dx)<\infty$, 
\end{enumerate}
where $\nu_R$ denotes the L\'evy measure of the  real-part of $(L_t)_{t\in \R}$. Furthermore, if  $X$ exists, then it has stationary increments, rotational invariant distribution and the characteristic function is given by
\begin{align*}
{}&\mathbb{E}\BBase*{\exp\Set*{i\left\langle \theta, \Big(\Re(X_t),\Im(X_t)\Big)\right\rangle}}
\\ {}&\qquad 
= \exp\PBase*{\int_\R K\bigg( \theta_1f_1(s)+\theta_2 f_2(s),\; \theta_2 f_1(s)-\theta_1 f_2(s) ,s\bigg) \, \lambda(ds)},
\end{align*}
for all  $\theta=(\theta_1,\theta_2)\in \R^2$, where $K$ is given by (\ref{align:Kfunc}).
\end{theorem}
 
 To prove Theorem~\ref{thm:hflm} we will first show the following lemma. 
 In this result, and in the following, we will write $f(t)\sim g(t)$ as $t \to a$ for real-valued functions $f$ and $g$, if  $\lim_{t\to a} (f(t)/g(t))=c$ for some constant $c\neq 0$. 

\begin{lemma} \label{lemma:asymptexistence}
Let $L$ be a real-valued symmetric Lévy process without Gaussian component and L\'evy measure $\nu$.  Let $f:\R\to\R$ be a measurable function, bounded on $[-1,1]^c$, and  satisfying 
\begin{equation}
\abs{f(s)}\sim \abs{s}^{\beta} \text{ as } s \to 0 \qquad \text{and}\qquad \abs{f(s)}\sim \abs{s}^{-\gamma} \text{ as } |s|\to  \infty, 
\end{equation}
for some $\beta\leq 0$ and $\gamma>0$. 
Then the stochastic integral $\int f \dif L$ exists if and only if the following two conditions (a) and (b) are satisfied:
\begin{enumerate}
\item \label{item:ii:lemma:asymptexistence} $\gamma > 1/2$ and the following condition hold true
\begin{align*}
\int_{\abs{x}>1}\abs{x}^{\frac{1}{\gamma}} \, \nu (dx) < \infty.
\end{align*}
\item \label{item:i:lemma:asymptexistence}  We have that 
\begin{align*}
\int_{\abs{x}\leq 1} \abs{x}^{\frac{1}{-\beta}} \, \nu (dx) < \infty.
\end{align*}
\end{enumerate}
\end{lemma}

If $\sim$ in Lemma~\ref{lemma:asymptexistence} is replaced by $f(s) = O(|s|^\beta)$ as $s\to 0$,  or 
$f(s)=O(|s|^{-\gamma})$ as $|s|\to \infty$,  the criteria for  existence of the integral $\int f\,dL$ remain sufficient. Note that if $\beta > -1/2$, the second criterion holds for any Lévy measure.

\begin{proof}[Proof of Lemma \ref{lemma:asymptexistence}]
Writing out the conditions in Theorem 2.7 of \cite{RajputRosinski1989} and observing that these are increasing in the function $f$, it suffices to study these conditions for a function $g(s)\coloneqq \mathds{1}_{\interval{-1}{1}}(s)\abs{s}^{\beta} + \mathds{1}_{\interval{-1}{1}^c}(s) \abs{s}^{-\gamma}$. Recall that the general condition for existence of $\int g \dif L$ is given by
\begin{align} \label{align:levyintegralexist}
\int_{\R} \int_{\R} \min (1, \abs{xg(s)}^2) \, \nu(dx) \, \lambda (ds)<\infty,
\end{align}
where $\nu$ denotes the Lévy measure of $L$ and $\lambda$ denotes the Lebesgue measure. Divide this condition into the following four areas, 
\begin{align*}
A_{11}&=\Set{(s,x)\in \R\times \R: \abs{s}\leq 1,\ \abs{x}\leq 1},
\\
A_{12}&=\Set{(s,x)\in \R\times \R: \abs{s}\leq 1,\ \abs{x}> 1},
\\
A_{21}&=\Set{(s,x)\in \R\times \R: \abs{s}> 1,\ \abs{x}\leq 1},
\\
A_{22}&=\Set{(s,x)\in \R\times \R: \abs{s}> 1,\ \abs{x}> 1}.
\end{align*}
The monotonicity of $g$ on these sets can then be used to simplify the condition in (\ref{align:levyintegralexist}) into  \ref{item:ii:lemma:asymptexistence} and \ref{item:i:lemma:asymptexistence}. 
We first consider $A_{22}$ and let $x\in \interval{-1}{1}^c$ be given. Divide the inner integral into
\begin{align*}
&\int_{\Set*{\abs{s}  > 1} \cap \Set{ \abs{s} > \abs{x}^{1/\gamma}}} \abs{x}^2 \abs{s}^{-2\gamma} \, \lambda(ds) +\int_{\Set{\abs{s}>1} \cap \Set{\abs{s} \leq \abs{x}^{1/\gamma}}} 1 \, \lambda(ds)
\\ & \qquad 
= \abs{x}^2 \int_{\abs{x}^{1/\gamma}}^{\infty} \abs{s}^{-2\gamma} \, \lambda(ds)   +2\lambda \PBase{\interval[open left]{1}{\abs{x}^{1/\gamma}}},
\\ & \qquad 
= \abs{x}^2 \BBase*{\frac{2}{-2\gamma +1}s^{-2\gamma+1}}_{\abs{x}^{1/\gamma}}^{\infty} + 2(\abs{x}^{1/\gamma}-1)
\\ & \qquad 
= \abs{x}^2 \abs{x}^{-2 + 1/\gamma} \frac{-2}{-2\gamma +1} + 2(\abs{x}^{1/\gamma}-1)= 3 \abs{x}^{1/\gamma} -2,
\end{align*}
where we have used that  $\gamma > 1/2$ to ensure the finiteness of the integral and afterwards that $\frac{2}{-2\gamma+1}<0$.
Inserting the derived into the original criterion on the set $A_{22}$, we get that
\begin{align*}
\int_{\abs{x}>1} \BBase*{3\abs{x}^{1/\gamma} - 2} \, \nu (dx) < \infty.
\shortintertext{Since the area $\abs{x}>1$ is of finite $\nu$-measure, this reduces to}
\int_{\abs{x}> 1} \abs{x}^{1/\gamma} \, \nu (dx) < \infty,
\end{align*}
which is one of the stated criterions. For $A_{11}$, let $x\in \interval{-1}{1}$ be given and assume that $\beta < 0$. The inner integral can in this case be written as
\begin{align*}
& \int_{\Set{\abs{s} \leq 1 }\cap \Set{\abs{s}^{\beta}\leq \abs{x}^{-1}}} \abs{xg(s)}^2 \, \lambda (ds) + \int_{\Set{\abs{s} \leq 1}\cap \Set{\abs{s}^{\beta} > \abs{x}^{-1}}}\lambda (ds)
\\
={}&  \abs{x}^2 \int_{1 \geq  \abs{s} \geq \abs{x}^{-1/\beta}} \abs{s}^{2\beta} \, \lambda(ds) +  \int_{\Set{\abs{s}\leq 1}\cap \Set{\abs{s} \leq \abs{x}^{-1/\beta}}} \lambda (ds)
\\
={}& \abs{x}^2 \frac{2}{2\beta+1}\BBase*{s^{2\beta +1}}_{\abs{x}^{-1/\beta}}^{1} + 2 \lambda(\interval{0}{\abs{x}^{-1/\beta}}).
\end{align*}
Inserting this into the outer integral we obtain
\begin{align*}
\int_{\abs{x} \leq 1}\PBase*{\abs{x}^2 \frac{2}{2\beta+1}\BBase*{s^{2\beta +1}}_{\abs{x}^{-1/\beta}}^{1} + 2 \lambda(\interval{0}{\abs{x}^{-1/\beta}})} \, \nu (dx)
\end{align*}
which reduces to the second condition by applying the definition of a Lévy measure. For $\beta = 0$, the proof is trivial. For $A_{12}$, let $x \in \interval{-1}{1}^c$ be given. We can again rewrite the inner integral into
\begin{align*}
&\int_{\Set{\abs{s} \leq 1}\cap \Set{\abs{s}^{-\gamma}\leq \abs{x}^{-1}}}  \abs{x}^2 \abs{s}^{-2\gamma}\lambda(ds)+ \int_{\Set{\abs{s} \leq 1}\cap \Set{\abs{s}^{-\gamma} > \abs{x}^{-1}}} \lambda (ds)
\\
={}& \int_{\Set{\abs{s} \leq 1} \cap \Set{\abs{s} \geq \abs{x}^{1/\gamma}}} \abs{x}^2 \abs{s}^{-2\gamma} \, \lambda(ds)  +  \int_{\Set{\abs{s} \leq 1} \cap \Set{\abs{s} < \abs{x}^{1/\gamma}}} \lambda (ds)
\\
={}& 0 + \lambda(\interval{0}{1}) ,
\end{align*}
where we used that $\abs{x} > 1$. Inserting this into the outer integral reduces to a trivial condition for Lévy measures. For the last area, $A_{21}$, let $x\in \interval{-1}{1}$ be given. In this case the condition similarly reduces to
\begin{align*}
\int_{\abs{x} \leq 1} \abs{x}^{2} \, \nu (dx) < \infty,
\end{align*}
which is trivial. This concludes the proof.
\end{proof}

\begin{proof}[Proof of Theorem~\ref{thm:hflm}]
Let $f$ denote the integrand of the harmonizable fractional L\'evy motion. Observe that $f$ is bounded on $[-1,1]^c$, and 
\begin{align*}
f(s)=O(\abs{s}^{-H-1/\alpha}) \text{ as }|s|\to \infty, \quad \text{and}  \quad f(s)=O(\abs{s}^{1-H-1/\alpha}) \text{ as }s\to 0.
\end{align*}
The existence criteria  now follows by Lemma~\ref{lemma:asymptexistence}. The stationary increments follows by a straightforward extension of Theorem 4.1 in \cite{urbanik1968random} to continuous time, see also Theorem 6.5.1 in \cite{SamoTaqqu1994} for the stable case. The isotropic distribution follows immediately from $(c)$ in Theorem \ref{thm:CFforcomplexint}.
\end{proof}

We are now ready to complete the proof of our main result.
\begin{proof}[Proof of Theorem \ref{thm:lass}]
We study the characteristic function of the finite dimensional distributions for the left-hand side of (\ref{align:LASSpropertyofHFLM}) and show convergence towards the characteristic function of harmonizable fractional stable motion.  The characteristic function for the finite dimensional distribution of (\ref{align:LASSpropertyofHFLM}) is given  by Theorem~\ref{thm:CFforcomplexint}. For $(\theta^{(1)}_j,\theta^{(2)}_j)\in \R^2$ for $j=1,\dots,n$, we have that 
\begin{align} \label{align:lassCF}
 A_\epsilon:={}& \log \Aver*{\exp\PBase*{i\sum_{j=1}^n \BBase*{\theta_j^{(1)} \frac{\Re(X(\epsilon t_j))}{\epsilon^H}+\theta_j^{(2)} \frac{\Im(X(\epsilon t_j))}{\epsilon^H}}}} \nonumber
\\
={}& \int_\R \psi \Bigg( \epsilon^{-H} \sum_{j=1}^n \theta_j^{(1)} f_{\epsilon t_j,1}(s) +  \epsilon^{-H}\sum_{j=1}^n \theta_j^{(2)} f_{\epsilon t_j,2}(s), \nonumber 
\\
&\quad  \quad \quad  \epsilon^{-H} \sum_{j=1}^n \theta_j^{(2)} f_{\epsilon t_j,1}(s) - \epsilon^{-H} \sum_{j=1}^n \theta_j^{(1)} f_{\epsilon t_j,2}(s)\Bigg) \dif s,
\end{align}
where $f_{\epsilon t_j,1}, f_{\epsilon t_j,2}$ denotes the real, respectively imaginary, part of integrand $f_{\epsilon t_j}$ for $X_{\epsilon t_j}$ and with $z=(z_1,z_2)$ 
\begin{align*}
\psi (z_1,z_2) \coloneqq \int_{\R^2} \BBase*{e^{i\langle z,x\rangle} -1 -\mathds{1}_{\Set{\Vert x \Vert \leq 1}}(x)\langle z , x \rangle}\, \nu (dx)
\end{align*}
 Writing $u=\epsilon s$, we substitute the $\epsilon$ out of the time index of $f$ and obtain
\begin{align*}
f_{\epsilon t}(s) = \frac{e^{i\epsilon ts}-1}{is} \PBase*{a(s_+)^{-H-1/\alpha + 1}+b(s_-)^{-H-1/\alpha + 1}}= f_t(u) \epsilon^{H+1/\alpha}.
\end{align*}
Making the substitution $u=\epsilon s$ in equation (\ref{align:lassCF}) thus yields that
\begin{align*}
A_\epsilon =&\int_{\R} \psi \left( \epsilon^{H+1/\alpha - H} \PBase*{ \sum_{j=1}^n \theta_j^{(1)}  f_{t_j,1}(u) +  \sum_{j=1}^n \theta_j^{(2)} f_{t_j,2}(u)},  \right.
\\
&\phantom{\int_{\R} \psi \Big(\ }\left. \epsilon^{H+1/\alpha - H} \PBase*{\sum_{j=1}^n \theta_j^{(2)} f_{t_j,1}(u) - \sum_{j=1}^n \theta_j^{(1)} f_{t_j,2}(u)} \right) \epsilon^{-1}\dif u. 
\end{align*}
To simplify notation, define $g_{\theta,t}(u)\in \R^2$ by 
\begin{align} \label{align:gfunc}
\bigg( \Big( \sum_{j=1}^n \theta_j^{(1)}  f_{t_j,1}(u) +  \sum_{j=1}^n \theta_j^{(2)} f_{t_j,2}(u)\Big),  \Big( \sum_{j=1}^n \theta_j^{(2)} f_{t_j,1}(u) - \sum_{j=1}^n \theta_j^{(1)} f_{t_j,2}(u) \bigg), 
\end{align}
and let  $k(z,x)\coloneqq e^{i\langle z,x\rangle} -1 -i\mathds{1}_{D^c}(x) \langle z,x \rangle$. Inserting the defined notation, this implies that we may rewrite the characteristic function to
\begin{align} \label{align:defofh}
A_{\epsilon} = &\int_\R \int_{\R^2} k(\epsilon^{1/\alpha}g_{\theta,t}(u),x)\, \nu (dx) \epsilon^{-1}\dif u  \nonumber
\\
=&\int_\R \int_{\R^2} k(\epsilon^{1/\alpha}g_{\theta,t}(u),x)  \, f(x) \dif x \, \epsilon^{-1} \dif u \nonumber
\\
=&
\int_\R \int_{\R^2} k(g_{\theta,t}(u),x) f(\epsilon^{-1/\alpha}x) \epsilon^{2(-1/\alpha)} \dif x \dif \epsilon^{-1} \dif u,
\end{align}
where we used $\nu (dx)=f(x) \dif x$ and a simple scaling of parameters in $\R^2$.
The next step is to show pointwise convergence of the integrand as $\epsilon \to 0_+$. After this we apply the dominated convergence theorem to insert the found limit under the integral. We postpone the argument for dominated convergence theorem until the end of this proof. Assumption~(A) on the Lévy measure $\nu$ gives us that for every $\delta>0$ we can find $R_{\delta} > 0$ such that 
\begin{align*}
1-\delta \leq \frac{f(x)}{\Vert x \Vert^{-2-\alpha}}\leq 1+\delta, \quad \text{for }\Vert x \Vert \geq R_{\delta}.
\end{align*}
Fix $x\in \R^2\setminus \Set{0}$ and $u\in\R$. For every $\delta > 0$ we can choose $\epsilon$ sufficiently small such that $\Vert \epsilon^{-1/\alpha} x \Vert > R_{\delta}$, which implies that
\begin{align*}
1-\delta \leq\frac{f(\epsilon^{-1/\alpha}x)e^{-2/\alpha-1}}{\Vert x \Vert^{-2-\alpha}}=\frac{f(\epsilon^{-1/\alpha}x)}{\Vert \epsilon^{-1/\alpha}x\Vert^{-2-\alpha} }\leq 1+\delta.
\end{align*}
Thus in the limit we find that
\begin{align*}
&\lim_{\epsilon \downarrow 0} f(\epsilon^{-1/\alpha}x)\epsilon^{-2/\alpha -1} = \Vert x \Vert^{-2-\alpha},
\shortintertext{and hence }
&\lim_{\epsilon \downarrow 0} k(g_{\theta,t}(u),x)f(\epsilon^{-1/\alpha}x)\epsilon^{-2/\alpha -1} = \Vert x \Vert^{-2-\alpha}k(g_{\theta,t}(u),x).
\end{align*}
This finishes the proof of pointwise convergence for $f_\epsilon$ as $\epsilon \to 0$. 
Applying the dominated convergence theorem we find that
\begin{align} \label{align:CFofLASSlimit}
&\lim_{\epsilon \downarrow 0} \log \Aver*{\exp\PBase*{i\sum_{j=1}^n \BBase*{\theta_j^{(1)} \frac{\Re(Y(\epsilon t_j))}{\epsilon^K}+\theta_j^{(2)} \frac{\Im(Y(\epsilon t_j))}{\epsilon^K}}}} \nonumber
\\ &\qquad 
=\lim_{\epsilon \downarrow 0}\int_\R \int_{\R^2} k\PBase*{g_{\theta,t}(u),x}
f(e^{-1/\alpha}x) \epsilon^{-2/\alpha-1} \dif x \dif u \nonumber
\\   &\qquad 
= \int_\R \int_{\R^2} k\PBase*{g_{\theta,t}(u),x} \Vert x \Vert^{-2-\alpha} \dif x \dif u.
\end{align}
The book \cite{applebaum2009levy}, p.~37, identifies $\Vert x \Vert^{-2-\alpha}$ in (\ref{align:CFofLASSlimit}) as the Lévy measure of a rotationally invariant two-dimensional $\alpha$-stable Lévy process. We can continue our derivations in polar coordinates and observe that the inner integral may be rewritten as
\begin{align*} 
 {}&\int_{\R^2} k\PBase*{g_{\theta,t}(u),x} \Vert x \Vert^{-2-\alpha} \dif x  
\\
={}& \int_0^{2\pi}  \int_0^\infty k\PBase*{g_{\theta,t}(u),r(\cos(s),\sin(s))} \, r^{-1-\alpha} \dif r \dif s
\\
={}&  \int_{0}^{2\pi} -c_0 \abs{\langle g_{\theta,t}(u), \big( \cos(s),\sin(s)\big) \rangle}^{\alpha} \dif s.
\end{align*}
Here we used the following result, which follows by substituting $z=yr$,
\begin{align*}
- c_0 \abs{y}^{\alpha} = \int_0^\infty \BBase*{\exp(iyr)-1 -i \mathds{1}_{\Set{\Vert r \Vert \leq 1}}(x) yr} r^{-1-\alpha} \dif r,
\end{align*}
where $c_0\coloneqq \int_0^\infty \BBase{\cos(r)-1}\, r^{-1-\alpha} \dif r$. Write $g_{\theta,t}(u)=\Vert g_{\theta,t}(u)\Vert(\cos(\kappa_u),\sin(\kappa_u))$ in polar form for some $\kappa_u$. Inserting this notation and applying a standard trigonometric rule, we obtain
\begin{align*}
{}&  \int_{0}^{2\pi} -c_0 \abs{\langle g_{\theta,t}(u), \big( \cos(s),\sin(s)\big) \rangle}^{\alpha} \dif s.
\\
={} &  - c_0 \Vert g_{\theta,t}(u)\Vert^{\alpha} \int_0^{2\pi} \abs{\langle(\cos(\kappa_u),\sin(\kappa_u)), (\cos(s),\sin(s))\rangle}^{\alpha} \dif s
\\
={} & - c_0 \Vert g_{\theta,t}(u)\Vert^{\alpha} \int_0^{2\pi} \abs{\cos(\kappa_u) \cos(s) + \sin(\kappa_u)\sin(s)}^{\alpha} \dif s
\\
={} &  - c_0 \Vert g_{\theta,t}(u)\Vert^{\alpha} \int_0^{2\pi} \abs{\cos(s-\kappa_u)}^{\alpha} \dif s = - c_0 \Vert g_{\theta,t}(u)\Vert^{\alpha} c_1,
\end{align*}
where $c_1 = \int_0^{2\pi} \abs{\cos(s)}^{\alpha} \dif s$. Inserting this into~(\ref{align:CFofLASSlimit}), we identify the characteristic function as 
\begin{align*}
\exp \PBase*{ - c_0c_1 \int_{\R} \Vert g_{\theta,t}(u)\Vert^{\alpha} \dif s}
\end{align*}
which is the characteristic function of harmonizable fractional stable motion stated in Theorem 6.3.4 of \cite{SamoTaqqu1994} and on p.\ 359 of the same book when we insert $g_{\theta,t}$ (up to a scaling factor). 
Thus all that remains is the argument for dominated convergence theorem in equation~(\ref{align:CFofLASSlimit}). By assumption there exists a $C > 0$ such that $f(x) \leq C\Vert x \Vert^{-2-\alpha}$ for all $x\in\R$. This implies that 
\begin{align*}
f(\epsilon^{-1/\alpha}x)e^{-2/\alpha-1} \leq C \Vert \epsilon^{-1/\alpha} x \Vert^{-2-\alpha} \epsilon^{-2/\alpha -1}=C\Vert x\Vert^{-2-\alpha}.
\end{align*}
Thus a good candidate for a  dominating (integrable) function would be
\begin{align*}
F(x,u)=\Vert\PBase{g_{\theta,t}(u),x} \Vert  C \Vert x \Vert^{-2-\alpha}.
\end{align*}
From classical theory of Lévy measures, we know that
\begin{align*} &\abs{k\PBase{g_{\theta,t}(u),x}} \leq 1\wedge \BBase*{\Vert g_{\theta,t}(u) \Vert^2 \Vert x \Vert^2},
\intertext{which implies that}
&F(x,u) \leq C \PBase*{\Vert x \Vert^{-2-\alpha} \wedge \BBase*{\Vert g_{\theta,t}(u) \Vert^2 \Vert x \Vert^{-\alpha} }}.
\end{align*}
By changing to polar coordinates we obtain that (the constant changes from line to line)
\begin{align*}
&\int_\R \int_{\R^2} C \PBase*{\Vert x \Vert^{-2-\alpha} \wedge \BBase*{\Vert g_{\theta,H,t}(u) \Vert^2 \Vert x \Vert^{-\alpha} }} \dif x \dif u
\\
& \qquad = \int_\R \int_0^{2\pi} \int_0^\infty C \PBase*{r^{-2-\alpha} \wedge \BBase*{\Vert g_{\theta,H,t}(u) \Vert^2 r^{-\alpha} }}  r \dif r \dif \psi \dif u
\\ & \qquad 
\leq  \sum_{j=1}^n C\int_\R \int_0^\infty  \PBase*{r^{-1-\alpha} \wedge \BBase*{\Vert f_{t_j}(u) \Vert^2 r^{-\alpha+1} }}   \dif r  \dif u,
\end{align*}
where $f_t$ denotes the integrand of the harmonizable fractional Lévy motion at time $t$. This is exactly the criterion for the existence of the stochastic integral $\int \abs{f_t} \dif \tilde{L}_s$ wrt. an $\alpha$-stable real-valued Lévy process $\tilde{L}$. By the choice of $(\alpha,H)\in \interval[open]{0}{2}\times \interval[open]{0}{1}$ such an integral exists by the existence of the harmonizable fractional stable motion for these parameters. This concludes the argument for dominated convergence and hence the proof.
\end{proof}

\bibliographystyle{plainnat}

\end{document}